\documentclass[11pt]{article}
\usepackage{geometry}                		% See geometry.pdf to learn the layout options. There are lots.
\usepackage[parfill]{parskip}    		% Activate to begin paragraphs with an empty line rather than an indent
\usepackage{graphicx}				% Use pdf, png, jpg, or eps§ with pdflatex; use eps in DVI mode
\usepackage{amssymb}
\usepackage{amsmath}
\usepackage{palatino}
\usepackage[english]{babel}
\usepackage{amscd}
\usepackage{amsthm}
\usepackage{color}

\newtheorem{theorem}{Theorem}[section]
\newtheorem{lemma}[theorem]{Lemma}

\newtheorem{definition}{Definition}

\newtheorem{problem}[theorem]{Problem}

\newcommand{\C}{\mathcal{C}}
\newcommand{\F}{\mathcal{F}}
\newcommand{\R}{\mathbb R}
\newcommand{\N}{\mathbb N}
\newcommand{\Ha}{\mathcal{H}}
\newcommand{\Le}{\mathcal{L}}
\newcommand{\eps}{\varepsilon}
\newcommand{\vv}{\boldsymbol{v}}

\newcommand{\1}{\boldsymbol 1}
\newcommand{\0}{\boldsymbol 0}
\newcommand{\BO}{\boldsymbol{O}}
\newcommand{\x}{\boldsymbol{x}}
\newcommand{\y}{\boldsymbol{y}}
\newcommand{\dd}{\boldsymbol{d}}
\newcommand{\e}{\boldsymbol{e}}
\newcommand{\ba}{\boldsymbol{a}}
\newcommand{\bb}{\boldsymbol{b}}

\title{A continuous analogue of Erd\H{o}s' $k$-Sperner theorem}

\author{
Themis Mitsis\thanks{Department of Mathematics and Applied Mathematics, University of Crete, 70013 Heraklion, Greece. 
E-mail: themis.mitsis@gmail.com}
\and
Christos Pelekis\thanks{Institute of Mathematics, Czech Academy of Sciences, \v{Z}itna 25, Praha 1, Czech Republic. Research supported by GA\v{C}R project 18-01472Y  and RVO: 67985840. E-mail: pelekis.chr@gmail.com}
\and
V\'aclav Vlas\'ak\thanks{ Faculty of Mathematics and Physics, Charles University, Sokolovsk\'a 83, 18675 Praha 8, Czech Republic. E-mail:  vlasakvv@gmail.com}
}

\begin{document}
\maketitle
%\date{\today}							% Activate to display a given date or no date

\begin{abstract}
A \emph{chain} in the unit $n$-cube is a set $C\subset [0,1]^n$ such that  for every 
$\mathbf{x}=(x_1,\ldots,x_n)$ and $\mathbf{y}=(y_1,\ldots,y_n)$ in  $C$ we either have  
$x_i\le y_i$ for all $i\in [n]$, or $x_i\ge y_i$ for all $i\in [n]$. We show that the $1$-dimensional Hausdorff measure of a chain in the unit $n$-cube is at most $n$, and that the bound is sharp. Given this result, we consider the problem of maximising the $n$-dimensional Lebesgue measure of a measurable set $A\subset [0,1]^n$ subject to the constraint that it satisfies 
$\Ha^1(A\cap C) \le \kappa$ for all chains $C\subset [0,1]^n$, 
where $\kappa$ is a fixed real number from the interval  $(0,n]$. We show that the measure of $A$ is not larger than the measure of the following optimal set:  
\[
A^{\ast}_{\kappa} = \left\{ (x_1,\ldots,x_n)\in [0,1]^n : \frac{n-\kappa}{2}\le \sum_{i=1}^{n}x_i \le  \frac{n+ \kappa}{2} \right\} \, .
\]
Our result may be seen as a continuous counterpart to a theorem of Erd\H{o}s, regarding $k$-Sperner families of finite sets. 
\end{abstract}

\noindent{
\emph{Keywords}: chains; $k$-Sperner families; Hausdorff measure; Lebesgue measure 
}
	
\noindent{
\emph{MSC (2010)}:     05D05; 28A78;  05C35
}

\section{Prologue, related work and main results}

Let $[n]$ denote the set of positive integers $\{1,\ldots,n\}$, and $2^{[n]}$ denote the power-set of $[n]$. A family $\C\subset 2^{[n]}$ is called a \emph{chain} if for every distinct $C_1,C_2\in \C$  we either have $C_1\subset C_2$ or $C_2\subset C_1$. We assume that the chains under consideration do not contain the empty set. 
Here and later, the cardinality of a finite set $F$ is denoted $|F|$. Let $k \in [n]$ be a positive integer. 
A family $\F\subset 2^{[n]}$ is 
called \emph{$k$-Sperner} if there is no chain $\C\subset\F$ such that $|\C|=k+1$.  In other words, a $k$-Sperner family is a collection $\F\subset 2^{[n]}$ such that $|\F\cap \C|\le k$, for all chains $\C\subset 2^{[n]}$.  Given two points $\mathbf{x}=(x_1,\ldots, x_n)$ and $\mathbf{y}=(y_1,\ldots,y_n)$ in $\R^n$, we write $\mathbf{x} \le \mathbf{y}$ if $x_i\le y_i$, for all $i\in [n]$. 

Let us begin with a well-known result of Erd\H{o}s, that  provides a sharp upper bound on the size of  $k$-Sperner families. 

\bigskip

\begin{theorem}[Erd\H{o}s~\cite{Erdos}]
\label{k_sperner}
Let $\F$ be a $k$-Sperner family of $2^{[n]}$.  Then the cardinality of $\F$ is not greater than the sum of the $k$ largest binomial coefficients. 
\end{theorem}

For $k=1$, Theorem~\ref{k_sperner} is due to Sperner (see~\cite{sperner}). The notion of   
$k$-Sperner families is fundamental in extremal set theory and has inspired a vast amount of research.  
We refer the reader to~\cite{Anderson, Engel_Sperner} for legible textbooks on the topic. In this article we shall be interested in a continuous analogue of Erd\H{o}s' result. 
It has been almost half a century (see~\cite{Bollobas, katona_one, katona_two, Nash-Williams}) since the idea was conceived that several results from extremal combinatorics have continuous counterparts. 
This idea has inspired several continuous analogues of results from extremal combinatorics both in a ``measure-theoretic setting'' (see, for example, \cite{Bollobas, Dolezal_Mitsis_Pelekis, engel, EMPR, katona_one, katona_two})  
 and in a ``vector space setting'' (see, for example, \cite{FranklWilson, klain_rota}). 
 In this article we investigate a continuous analogue of Theorem~\ref{k_sperner}. 
 Let us proceed by stating a result due to Konrad Engel~\cite{engel} that is similar to our main result. Here and later, $\Le^n(\cdot)$ denotes $n$-dimensional Lebesgue measure.  
 
 \bigskip
 
\begin{theorem}[Engel~\cite{engel}]
\label{th:01} Let $\kappa>0$ be a real number and let
$A$ be a Lebesgue measurable subset of $[0,1]^n$ that does not contain two elements $\mathbf{x}=(x_1,\ldots,x_n)$ and $\mathbf{y}=(y_1,\ldots,y_n)$ such that $\mathbf{x} \le \mathbf{y}$ and 
 $\sum_{i=1}^n(y_i-x_i) \geq \kappa$. 
Then the $n$-dimensional Lebesgue measure of $A$ is not greater than the measure of the following optimal set:
\[ A_{\kappa} :=\left\{(x_1,\ldots,x_n) \in [0,1]^n: \frac{n-\kappa}{2}\leq \sum_{i=1}^n x_i < \frac{n+\kappa}{2}\right\}. \] 
Moreover, if we set  
$v_n(\kappa):= 1 - \frac{2}{n!} \sum_{j=0}^{\lfloor \frac{n-\kappa}{2} \rfloor} (-1)^j \binom{n}{j} \left( \frac{n-\kappa}{2} -j \right)^n$,
where $\lfloor x\rfloor$ denotes the largest integer that is less than or equal to $x$, then we have $\Le^n(A_{\kappa}) = v_n(\kappa)$.  
\end{theorem}
 
Notice that the measure of set $A_{\kappa}$, in Theorem~\ref{th:01}, depends continuously on $\kappa$ and therefore $v_n(\kappa)$ is a continuous function of $\kappa$. 
  
Before stating our main results, let us proceed with some remarks. 
Notice that one can associate a binary vector $\mathbf{x}\in \{0,1\}^n$ to each subset $F$ of $[n]$: simply put $1$ in the $i$-th coordinate if $i\in F$, and $0$ otherwise. Notice that this correspondence is bijective and one may choose to not distinguish between subsets of $[n]$ and binary vectors of length $n$. 
Hence, another way to think of chains in $2^{[n]}$ is to consider subsets $C\subset \{0,1\}^n$ such that 
for every distinct $\mathbf{x}=(x_1,\ldots,x_n)$ and $\mathbf{y}=(y_1,\ldots,y_n)$ in  $C$ we either have  
$\mathbf{x}\le \mathbf{y}$, or $\mathbf{y}\le \mathbf{x}$. Clearly, the maximum size of a chain, which does not contain the empty set, is at most $n$. 
Given the aforementioned observations, Theorem~\ref{k_sperner} can be equivalently expressed as follows. 

\bigskip

\begin{theorem}[Theorem~\ref{k_sperner} restated]
\label{k_sperner2}
Fix a positive integer $k\in [n]$. 
Let $A\subset \{0,1\}^n$ be such that 
\[
|A\cap C|\le k, \; \text{ for all chains } \; C\subset \{0,1\}^n \, .
\]
Then $|A|\le \sum_{i=1}^{k} \binom{n}{\lfloor \frac{n-k}{2}\rfloor +i}$. 
\end{theorem}

It seems natural to ask what happens if one replaces the binary $n$-cube $\{0,1\}^n$ with the unit $n$-cube $[0,1]^n$ in Theorem~\ref{k_sperner2}. 
Bearing this in mind, we proceed with the following. 

\bigskip

\begin{definition}[Chains] 
A \emph{chain} is a set $C\subset \R^n$ such that  for every distinct  
$\mathbf{x},\mathbf{y}\in C$ we either have  
$\mathbf{x}\le \mathbf{y}$, or $\mathbf{y}\le \mathbf{x}$.   
\end{definition}

An example of a chain in the unit $n$-cube is the set 
\[
C = \{ ( f_1(x),\ldots, f_n(x) ): x\in [0,1] \} \, , 
\]
where, for $i\in [n]$, $f_i: [0,1]\to [0,1]$ is a non-decreasing function.  

What is the maximum ``size'' of a chain in the unit $n$-cube? Since we are dealing with subsets of the unit $n$-cube we have to choose a suitable notion of ``size''. A first choice could be the $n$-dimensional Lebesgue measure.
However, it is not difficult to see, using Lebesgue's density theorem, that the Lebesgue measure of a chain in the unit $n$-cube equals zero. Given this observation, it is then natural to ask for sharp upper bounds on the Hausdorff dimension and the corresponding Hausdorff measure of chains in the unit $n$-cube. Our first result provides best possible bounds on both quantities.  
Throughout the text,  given $s\in [0,\infty)$, $\Ha^s(\cdot)$ denotes $s$-dimensional Hausdorff outer measure (see~\cite[p.~81 and p.~1--2]{Evans_Gariepy}). 

\bigskip

\begin{theorem}\label{chains}
Let $C\subset [0,1]^n$ be a chain. Then $\Ha^1(C)\le n$. 
\end{theorem}

The bound provided by Theorem~\ref{chains} is best possible, as can be seen from the chain 
\[
C = \bigcup_{i=1}^{n}\{(x_1,\ldots,x_n)\in [0,1]^n  : x_1 = \cdots = x_{i-1} =1 , \, x_{i+1} = \cdots =x_n=0\} \, .
\]
A more ``exotic'' example of a chain in the unit $n$-cube whose $1$-dimensional Hausdorff measure equals $n$ can be found in the proof of~\cite[Theorem~1.5]{Dolezal_Mitsis_Pelekis}.  
Now, given Theorem~\ref{chains} and Theorem~\ref{k_sperner2}, it seems natural to ask for upper bounds on the maximum ``size'' of a subset of the unit $n$-cube whose intersection with every chain has $\Ha^1$-measure which is not larger than a given number from the interval $(0,n]$. This leads to the following continuous analogue of Erd\H{o}s' theorem. Throughout the text, the term \emph{measurable set} refers to a set that is Lebesgue measurable. 

\bigskip

\begin{theorem}\label{cont_k_sperner}
Fix a real number $\kappa\in (0,n]$. Let $A\subset [0,1]^n$ be a measurable set that satisfies 
\[
\Ha^1(A\cap C)\le \kappa, \; \text{ for all chains } \; C\subset [0,1]^n \, .
\] 
Then the $n$-dimensional Lebesgue measure of $A$ is not greater than the measure of the following optimal set: 
\[ A_{\kappa}^{\ast} :=\left\{(x_1,\ldots,x_n) \in [0,1]^n: \frac{n-\kappa}{2}\leq \sum_{i=1}^n x_i \le \frac{n+\kappa}{2}\right\} \, . \] 
Moreover, we have 
$\Le^n(A_\kappa^{\ast})= v_n(\kappa)$, where $v_n(\kappa)$ is as in Theorem~\ref{th:01}.
\end{theorem}

\subsection{Organisation}

In Section~\ref{sec:chains} we prove Theorem~\ref{chains} by showing that the $\Ha^1$-measure of $A$ is less than or equal to the sum of the  $\Ha^1$-measures of its $n$ ``anti-diagonal'' projections onto the $n$ coordinate axes. 
Sections~\ref{sec:k_sperner} and~\ref{vasek} are devoted to the proof of Theorem~\ref{cont_k_sperner}. The proof is based on, and is inspired from, the proof of Theorem~\ref{th:01} (see~\cite{engel}) and proceeds by discretising the problem and by employing well know results from the theory of (finite) partially ordered sets. 
Finally, in Section~\ref{conclusion} we collect some remarks and an open problem. 

\section{Proof of Theorem~\ref{chains}}\label{sec:chains}

Given a chain $C\subset [0,1]^n$ and $i\in [n]$, let $C^{(i)}$ denote the set 
\[
C^{(i)} = C \cap \left\{(x_1,\ldots,x_n)\in [0,1]^n : i-1\le \sum_{i=1}^n x_i \le i \right\} \, . 
\]
Moreover, given $\mathbf{x}\in C^{(i)}$, let $S_i(\mathbf x)= (\sum_{j=1}^{n}x_j) - (i-1)$. 
For each $i\in [n]$ consider the "anti-diagonal" projections $\vartheta_i : C^{(i)}\to [0,1]^n$  
defined by  
\[
(x_1,\ldots,x_n) \mapsto (0,\ldots,0, S_i(\mathbf x),0,\ldots,0)\, ,
\]
where $S_i(\mathbf x)$ is on the $i$-th coordinate. Notice that, for each $i\in [n]$, $\vartheta_i$ restricted on $C^{(i)}$ is injective and therefore is a bijection from $C^{(i)}$ onto its image $\vartheta_i(C^{(i)})$. Let $\mathbf{a},\mathbf{b}\in \vartheta_i(C^{(i)})$ be distinct and suppose that $\vartheta_i^{-1}(\mathbf{a}) = \mathbf{x}$ and $\vartheta_i^{-1}(\mathbf{b}) = \mathbf{y}$, for some $\mathbf{x},\mathbf{y}\in C^{(i)}$. 
Suppose, without loss of generality, that $S_i(\mathbf{x}) \ge S_i(\mathbf{y})$. 
Now notice that 
\[
\| \vartheta_i^{-1}(\mathbf{a})- \vartheta_i^{-1}(\mathbf{b}) \| = \sqrt{ \sum_{i=1}^{n} (x_i - y_i)^2 } \le \sum_{i=1}^{n} (x_i - y_i) = S_i(\mathbf{x}) -  S_i(\mathbf{y}) = \| \mathbf{a} - \mathbf{b}\| \, .
\]
This implies that, for each $i\in [n]$, the function $\vartheta_i^{-1}: \vartheta_i(C^{(i)})\to C^{(i)}$ is Lipschitz with constant $1$ and therefore (see~\cite[Theorem~2.8 ]{Evans_Gariepy}) we have   
\[
\Ha^1(C^{(i)}) = \Ha^1( \vartheta_i^{-1}(\vartheta_i(C^{(i)})) ) \le \Ha^1(\vartheta_i(C^{(i)}))  \, . 
\]
Hence 
\[ 
\Ha^1(C) \le \sum_{i=1}^{n} \Ha^1(\vartheta_i(C^{(i)})) 
\]
and, since we clearly have $\Ha^1(\vartheta_i(C^{(i)}))\le 1$,   the result follows.

\section{Proof of Theorem~\ref{cont_k_sperner}}\label{sec:k_sperner}

In this section we prove Theorem~\ref{cont_k_sperner}. 
The proof  requires some extra piece of notation. 
Throughout this section,  $[m-1]_0$ denotes the set of integers $\{0,1,\ldots,m-1\}$.
Given positive integers $j$ and $m\ge 2$ such that $j<m$, we denote by $I_{j,m}$ the intervals 
\[
I_{j,m}=
\begin{cases}
[\frac{j}{m},\frac{j+1}{m})\, , &\text{ if } j \in [m-2]_0,\\
[\frac{j}{m},\frac{j+1}{m}] \, , &\text{ if } j =m-1.
\end{cases}
\]
The approach we embark on is based on, and is inspired from, the approach in~\cite{engel}. In particular, we make use of the following result from~\cite[Lemma~2]{engel}.
 
\bigskip

\begin{lemma}[\cite{engel}]
\label{engel_lemma}
Let $V_{n,m}(\kappa)$ be the sum of the $\lceil \kappa m +n \rceil$ largest coefficients in the polynomial 
$p(x) = (1+ x + \cdots  + x^{m-1})^n$.
Then we have $\lim_{m\to\infty} \frac{V_{n,m}(\kappa)}{m^n} = v_n(\kappa)$, where $ v_n(\kappa)$ is defined in Theorem~\ref{th:01}. 
\end{lemma}

The sum of the $k$ largest coefficients in the polynomial 
$p(x) = (1+ x + \cdots  + x^{m-1})^n$ are also referred to as the $k$ \emph{largest Whitney numbers} of $[m-1]_{0}^{n}$ (see \cite[p.~25]{Greene_Kleitman}). 

We now proceed with the proof of Theorem~\ref{cont_k_sperner}. 
Let $A\subset [0,1]^n$ be a measurable set that satisfies $\Ha^1(A\cap C)\le \kappa$, for all chains $C\subset [0,1]^n$. Notice that Theorem~\ref{th:01} implies that it is enough to show 
\begin{equation}\label{eq_0}
\Le^n(A)\le v_n(\kappa) \, , 
\end{equation}
where $v_n(\kappa)$ is as in Theorem~\ref{th:01}. Moreover, the inner regularity of Lebesgue measure implies that it is enough to assume that $A$ is 
\emph{compact}. 

If $\kappa =n$, then Theorem~\ref{chains} implies that the unit $n$-cube has maximum $\Le^n$-measure. We may therefore assume that $\kappa < n$. 

Fix $\eps < \frac{1}{2n+2}$ which is additionally assumed to be sufficiently small so that it satisfies  
\begin{equation}\label{eps_condition}
\kappa < n(1-(2n+2)\eps)\cdot (1-\eps)^n \, .
\end{equation}
Write the unit $n$-cube $[0,1]^n$ as a union 
of cubes of the form
\[
Q_{\mathbf{d}}:=I_{d_1,m} \times I_{d_2,m} \times \dots \times I_{d_n,m},
\]
where  $\mathbf{d}=(d_1,\ldots,d_n)\in [m-1]_0^n$. Notice that each cube $Q_{\mathbf{d}}$ 
can be uniquely identified by the vector $\mathbf{d}\in [m-1]_0^n$. 
Given $F\subset [m-1]_0^n$, we denote  
\[
Q_F := \bigcup_{\mathbf{d}\in F} Q_{\mathbf{d}} \, .
\]
Consider the set of $n$-tuples 
\[
D_A = \{\mathbf{d}\in [m-1]_0^n : Q_{\mathbf{d}} \cap A \neq \emptyset \}\, .
\] 
Notice that $A\subset Q_{D_A}$. Moreover,  since $A$ is compact, 
we may assume that $m$ is large enough so that it holds 
\begin{equation}\label{equa:1}
\mathcal{L}^n(Q_{D_A}\setminus A) <\eps^{2n+1}  .
\end{equation}
Now consider the set  
\[
D_{\eps} = \{\mathbf{d}\in D_A :  \mathcal{L}^n(Q_{\mathbf{d}} \cap A ) > (1 - 2^{-n} \eps^{2n})\cdot \mathcal{L}^n(Q_{\mathbf d})  \} \, .
\]

The next lemma provides an upper bound on the maximum size of a chain in $D_{\eps}$. 

\bigskip

\begin{lemma}\label{D_e_chain} 
Let $t$ be the maximum size of a chain in $D_{\eps}$. Then $t\le \lceil m\kappa' + n\rceil$, where 
$\kappa' =\frac{\kappa}{(1-(2n+2)\eps) \cdot (1-\eps)^n}$. 
\end{lemma}

The proof of Lemma~\ref{D_e_chain} is rather technical and is deferred to Section~\ref{vasek}.  
For the remaining part of this section, let us assume that Lemma~\ref{D_e_chain} holds true. 
Notice that~\eqref{eps_condition} guarantees that $\kappa' < n$. 

Now Lemma~\ref{D_e_chain} implies that $D_{\eps}$ is a partially ordered set 
that does not contain a chain of length 
$\lceil m\kappa' +n \rceil$ and therefore (see~\cite[Theorem 5.1.4 and Example 5.1.1]{Engel_Sperner}) 
it follows that $|D_{\eps}|$ is not larger than the sum of the $\lceil m\kappa' +n \rceil$ largest Whitney numbers of $[m-1]_{0}^{n}$, i.e., we have 
\begin{equation}\label{equa:3}
|D_{\eps}| \le V_{n,m}(\kappa') \, , 
\end{equation}
where $V_{n,m}(\kappa')$ is defined in Lemma~\ref{engel_lemma}. 

\textbf{Claim}: We have $\Le^n(A) \le \Le^n(Q_{D_{\eps}}) + \frac{1 -  2^{-n} \eps^{2n}}{2^{-n} } \cdot \eps$. 
\begin{proof}[Proof of Claim]
Notice that \eqref{equa:1} implies $\Le^n(Q_{D_A}) < \Le^n(A) + \eps^{2n+1}$. 
Moreover, the definition of $D_{\eps}$ implies $\Le^n(Q_{D_A\setminus D_{\eps}} \cap A ) \le (1-  2^{-n} \eps^{2n}) \cdot \Le^n(Q_{D_A\setminus D_{\eps}})$. Therefore, we have 
\begin{eqnarray*}
\Le^n(A) &=& \Le^n(A\cap Q_{D_{\eps}}) + \Le^n(A\cap Q_{D_A\setminus D_{\eps}}) \\ 
&\le& \Le^n(Q_{D_{\eps}}) + (1-  2^{-n} \eps^{2n}) \cdot (\Le^n(Q_{D_A}) - \Le^n(Q_{D_{\eps}}) ) \\
&\le&  \Le^n(Q_{D_{\eps}}) + (1-  2^{-n} \eps^{2n}) \cdot (\Le^n(A) + \eps^{2n+1} - \Le^n(Q_{D_{\eps}}) )
\end{eqnarray*}
which in turn implies 
\[
\Le^n(A)  \le \Le^n(Q_{D_{\eps}}) + \frac{ (1-  2^{-n} \eps^{2n})\cdot \eps^{2n+1}}{ 2^{-n} \eps^{2n}} \, ,
\]
as desired. 
\end{proof}

Set $\delta(\eps) :=  \frac{1 -  2^{-n} \eps^{2n}}{2^{-n} } \cdot \eps$. 
To finish the proof of Theorem~\ref{cont_k_sperner}, notice that the Claim and \eqref{equa:3}  imply  
\begin{eqnarray*}
\Le^n(A) &\le& \Le^n(Q_{D_{\eps}}) + \delta(\eps) = \frac{|D_{\eps}|}{m^n} + \delta(\eps) \le  \frac{V_{n,m}(\kappa')}{m^n} + \delta(\eps) 
\end{eqnarray*}
and therefore, using Lemma~\ref{engel_lemma}, we conclude  
\[ 
\Le^n(A) \le v_n(\kappa') + \delta(\eps) \, . 
\]
The continuity of $v_n(\cdot)$ implies~\eqref{eq_0}, upon letting $\eps\to 0$, and  thus Theorem~\ref{cont_k_sperner} follows.

\section{Proof of Lemma~\ref{D_e_chain}}\label{vasek}

This section is devoted to the proof of Lemma~\ref{D_e_chain}. We begin by introducing some additional  piece of notation. 

We denote by $\1_n=(1,\ldots,1)$ the point in $\R^n$ all of whose $n$ coordinates are equal to $1$, and we occasionally drop the index when the underlying dimension is clear from the context.   
Given a positive integer $k\in [n]$, we denote by $\binom{[n]}{k}$ the family consisting of all subsets of $[n]$ whose cardinality equals $k$ and, 
given $F =\{i_1<\cdots <i_k\}\in \binom{[n]}{k}$, we let $\pi_F(\cdot)$ denote the function $\pi_F:\mathbb{R}^n \to\mathbb{R}^k$ which maps every point  $(x_1,\ldots,x_n)$ to the point $(x_{i_1},\ldots,x_{i_k})$. 
That is, $\pi_F(\cdot)$ is the projection  onto the coordinates corresponding to $F$.  
Moreover, given $\vv\in\R^n$ and $F\subset [n]$, we denote by $\vv_F$ the vector in $\R^n$ whose $i$-th coordinate equals $v_i$ if $i\in F$, and $0$ otherwise. For $i\in [n]$, we denote by $\e^i$ the $i$-th basis vector, i.e., the vector $(0,\ldots,0,1,0,\ldots,0)$ whose $i$-th coordinate equals $1$. 
If $A,B\subset \R^n$, we write $A\le B$ if for every $\x\in A$ and every $\y\in B$ we have $\x\le\y$. 
Given $\ba,\bb\in \R^n$ such that $\ba\leq\bb$, we denote by  $R_{\ba,\bb}\subset\R^n$ the rectangle 
$$R_{\ba,\bb}:=\{\x\in\R^n;\ \ba\leq\x\leq\bb\}.$$
If $\bb = \ba + \1_n$, we simply write $R_{\ba}$ instead of $R_{\ba,\bb}$. Thus $R_{\0}$ is another way to denote the set $[0,1]^n$. 
Finally, suppose we are given $k<n$ and $M\in \binom{[n]}{k}$,  a point $t\in \R^{n-k}$ and a set $A\subset \R^n$. Then we define 
\[A(t,M) := A \cap \pi_{[n]\setminus M}^{-1}(\{t\}) \, . \] 

We begin with a simple consequence of Fubini's theorem. 

\bigskip

\begin{lemma}\label{one cube}

Let $n\in \N$ and $\vv =(v_1,\dots,v_n)\in\R^n$ be fixed.  
Suppose that $\eps>0$ is such that for every $j\in[n]$ we have $\frac{1}{v_j+1}\geq\eps$ and let  $A\subset [0,1]^n$ be a compact set that satisfies 
\begin{equation}\label{L1hh}
\Le^n(A)>1-\eps^{n+1}.
\end{equation}
Then, for every $i\in [n]$, there exists a chain $V\subset [0,1]^n$ such that $\Ha^1(A\cap V)\geq 1-(v_i+1)\eps $ and 
\begin{equation}\label{L1hhh}
\{\eps\vv\}\leq V\leq   \{\eps (\vv+\1_n)_{[n]\setminus \{i\}}+\e^i\}  \, .
\end{equation}

\end{lemma}
\begin{proof}
Fix $i\in n$ and denote by $\vv^{(i)}$ the vector $\pi_{[n]\setminus \{i\}}(\vv)$. . 
Consider the rectangle $R :=\eps R_{\vv^{(i)}}$. 
Since $\frac{1}{v_j+1}\geq\eps$ for every $j\in[n]$ we have $R\subset [0,1]^{n-1}$. We now show that there exists $t_0\in R$ such that $\Ha^1(A(t_0,\{i\}))\geq 1-\eps$. Assume, towards a contradiction, that  there does not exist such a $t_0\in R$. Then for every $t\in R$ we have
\begin{equation}\label{L1h}
\Ha^1\big([0,1]^n\setminus A(t,\{i\})\big) = \Ha^1\big(([0,1]^n\setminus A)(t,\{i\}) \big) >\eps.
\end{equation}
By \eqref{L1hh}, \eqref{L1h}, Fubini's theorem and the fact that $R\subset [0,1]^{n-1}$ and $\Le^{n-1}(R)=\eps^{n-1}$, we conclude 
\begin{equation*}
\begin{aligned}
\eps^{n+1}>\Le^n([0,1]^n\setminus A)\geq\Le^n\left(\bigcup_{ t \in R}([0,1]^n \setminus A)(t,\{i\})\right)\geq\eps\cdot \Le^{n-1}(R)=\eps^{n+1},
\end{aligned}
\end{equation*}
which is a contradiction. Hence there exists $t_0\in R$ such that $\Ha^1(A(t_0,\{i\}))\geq 1-\eps$. 
Now consider the set 
$$V:=A(t_0,\{i\})\cap R_{\eps\vv,\1_n} \, .$$
Since $t_0\in R$ and $V\subset R_{\eps\vv,\1_n}$ we readily obtain (\ref{L1hhh}). Clearly $V$ is chain and, since $V=A\cap V$,  we conclude
$$\Ha^1(A\cap V)=\Ha^1(V)\geq\Ha^1(A(t_0,\{i\}))-\eps v_i\geq 1-(v_i+1)\eps\, ,$$
as desired.
\end{proof}

Given $\ba=(a_1,\ldots,a_n),\bb=(b_1,\ldots,b_n)\in \R^n$, let $\BO(\ba,\bb) = \{l\in [n]: a_l \neq b_l\}$ be the set of indices for which the corresponding coordinates are different.

\bigskip

\begin{lemma}\label{Cubes}
Let $n,k\in\N$ be fixed. 
Let $\vv=(v_1,\dots,v_n)\in\R^n_{+}$ be given and assume that  
$\dd^1 \le\cdots \le \dd^k\in\N^n$ is a chain of vectors with non-negative integer coordinates. 
Set $\alpha:=\max\{2n,v_1,\dots,v_n\}$ and $\BO:=\BO(\dd^1,\dd^k)$, and fix $\eps\in (0,\frac{1}{\alpha +2})$. 
Let  $A\subset\R^n$ be a compact set and assume further that there exists 
$W\subset[k]$ such that for every $i\in W$ we have 
\begin{equation*}
\Le^n(A\cap R_{\dd^i})>1-\eps^{2n}2^{-n} \, .
\end{equation*}
Then there exists a chain $L\subset \R^n$ such that 
\begin{equation}\label{inequality}
\Ha^1(A\cap L)\geq(1-(\alpha+2)\eps)(1-\eps)^n(|W|-1)
\end{equation}
and 
\begin{equation}\label{inequality_2}
\{\dd^1+\eps\vv\}\leq L\leq\{\dd^k+\eps(2|\BO|\1_{\BO}+  (\vv+\1)_{[n]\setminus\BO})\}.
\end{equation}
\end{lemma}

\begin{proof}

If $|W|\leq1$, then we may choose $L=\emptyset$ and the result follows. 
So we may sssume that $k\geq|W|>1$. In particular, we have $\BO\neq\emptyset$. 
Notice that the assumptions imply that $\eps$ satisfies $\eps(\vv+2\cdot\1_n)\leq\1_n$, a fact that will be used several times in the proof.

We proceed by induction on the dimension, $n$. 
The case $n=1$ is trivial; we may choose $L$ to be the interval $[\dd^1+\vv\eps,\dd^k]$.
Now, assuming that the lemma holds true for every integer less than or equal to $n-1$,  we prove it for $n$. We distinguish two cases.

First assume that $|\BO|<n$. Clearly, we have $|\BO|>0$ and 
\begin{equation}\label{Same projection}
\pi_{[n]\setminus \BO}(\dd^l)=\pi_{[n]\setminus \BO}(\dd^1)\, , \, \text{ for all } \, l\in [k] \, . 
\end{equation}
To simplify notation, set $\ba := \pi_{[n]\setminus \BO}(\dd^1+\eps\vv)$, $\bb := \pi_{[n]\setminus \BO}(\dd^1+\eps(\vv+\1))$ 
and $C:=\bigcup_{j\in W}R_{\dd^j}$. Notice that $\Le^n(C)=|W|$. 
By (\ref{Same projection}) we have
\begin{equation}\label{Size of section}
\Le^n\left(\bigcup_{t\in R_{\ba,\bb}}C(t,\BO)\right)=\Le^{n-|\BO|}(R_{\ba,\bb})\cdot \Le^n(C)=\eps^{n-|\BO|}\Le^{n}(C)\, .
\end{equation}
Since $\Le^n(A\cap R_{\dd^i} )>1-\eps^{2n}2^{-n}$ for every $i\in W$, we have 
$$\Le^n(C\setminus A)<\Le^{n}(C)\eps^{2n}2^{-n} \, ,$$
which, combined with \eqref{Size of section}, yields 
\begin{equation*}
\begin{aligned}
\Le^n\left(\bigcup_{t\in R_{\ba,\bb}}C(t,\BO)\setminus A\right)&\leq\Le^n(C\setminus A)<\Le^n(C)\eps^{2n}2^{-n}\\
&=\Le^n\left(\bigcup_{t\in R_{\ba,\bb}}C(t,\BO)\right)\eps^{n+|\BO|}2^{-n}\\
&\leq\Le^n\left(\bigcup_{t\in R_{\ba,\bb}}C(t,\BO)\right)\eps^{2|\BO|+1}2^{-|\BO|-1}.
\end{aligned}
\end{equation*}
Observe that $\pi_{\BO}:\pi_{[n]\setminus \BO}^{-1}(\{t\})\to\R^{|\BO|}$ preserves $\Ha^{|\BO|}$-measure. 
Moreover, notice that $\Le^{|\BO|}(\pi_{\BO}(C(t,\BO)))=|W|$.  Using Fubini Theorem we find $t\in R_{\ba,\bb}$ such that
\begin{equation}\label{Big section}
\Le^{|\BO|}(\pi_{\BO}(C(t,\BO)\setminus A))\leq \Le^{|\BO|}(\pi_{\BO}(C(t,\BO)))\eps^{2|\BO|+1}2^{-|\BO|-1}.
\end{equation}
Now, we find $\tilde{W}\subset W$ such that for every $l\in\tilde{W}$ we have
\begin{equation}\label{Nice cubes}
\Le^{|\BO|}(\pi_{\BO}( R_{\dd^l} (t,\BO)\cap A))\geq 1-\eps^{2|\BO|}2^{-|\BO|}.
\end{equation}
We show that $|\tilde{W}|-1\geq(1-\eps)(|W|-1)$. Assume, for the sake of contradiction, that 
$$|W\setminus\tilde{W}|>\eps(|W|-1)\geq\eps\frac12|W|.$$
Then for every $l\in W\setminus\tilde{W}$ we have
$$
\Le^{|\BO|}(\pi_{\BO}( R_{\dd^l} (t,\BO)\setminus A))> \eps^{2|\BO|}2^{-|\BO|}.
$$
Thus
\begin{equation*}
\begin{aligned}
\Le^{|\BO|}(\pi_{\BO}(C(t,\BO)\setminus A))&\geq\Le^{|\BO|}\left(\pi_{\BO}\left(\bigcup_{l\in W\setminus\tilde{W}}R_{\dd^l}(t,\BO)\setminus A\right)\right)\\
&>|W\setminus\tilde{W}|\eps^{2|\BO|}2^{-|\BO|}>\eps \frac12|W|\eps^{2|\BO|}2^{-|\BO|}\\
&=\eps^{2|\BO|+1}2^{-|\BO|-1}\Le^{|\BO|}(\pi_{\BO}(C(t,\BO))),
\end{aligned}
\end{equation*}
which contradicts (\ref{Big section}). Since $|\BO|<n$ we may apply the induction hypothesis for $\pi_{\BO}(\vv)$, $|\BO|$, $k$, $\eps$, $\pi_{\BO}(\dd^1),\dots,\pi_{\BO}(\dd^k)$, $\tilde{W}$, $\pi_{\BO}(A(t,\BO))$ instead of $\vv$, $n$, $k$, $\eps$, $\dd^1,\dots,\dd^k$, $W$, $A$ and obtain a chain $\tilde{L}\subset\R^{|\BO|}$ such that 
\begin{equation*}
\begin{aligned}
\Ha^1(\tilde{L}\cap \pi_{\BO}(A(t,\BO)))
&\geq(1-(\tilde{\alpha}+2)\eps)(1-\eps)^{|\BO|}(|\tilde{W}|-1)\\
&\geq(1-(\alpha+2)\eps)(1-\eps)^{|\BO|+1}(|W|-1)\\
&\geq(1-(\alpha+2)\eps)(1-\eps)^n(|W|-1),
\end{aligned}
\end{equation*}
where $\tilde{\alpha}=\max\{2|\BO|,v_1,\dots,v_n\}$ and
\begin{equation*}
\{\pi_{\BO}(\dd^1+\eps\vv)\}\leq\tilde{L}\leq\{\pi_{\BO}(\dd^k+2|\BO|\eps\1)\}.
\end{equation*}
Now define the set $L:=\pi^{-1}_{\BO}(\tilde{L})\cap\pi^{-1}_{[n]\setminus \BO}(t)$. Clearly, $L$ is a chain. Since $t\in R_{\ba,\bb}$ we  obtain \eqref{inequality_2}. Finally,  the fact that $\pi_{\BO}$ is a linear isometry on $A (t,\BO)$ implies \eqref{inequality}. The proof of the first case is thus completed.  

Now assume that $|\BO|=n$. First we show that we may additionally suppose that $W=[k]$. 
Let $W=\{i_1,\ldots,i_p\}$, where $p>1$ and  $i_1<\dots<i_p$. Consider the sets 
\[
\tilde{\BO}:=\BO(\dd^{i_1},\dd^{i_p}), \, \tilde{\BO}_1:=\BO(\dd^1,\dd^{i_1})\setminus\tilde{\BO}\, \text{ and } \, \tilde{\BO}_2:=\BO(\dd^{i_p},\dd^{i_k})\setminus(\tilde{\BO}\cup\tilde{\BO}_1) \, .
\]

Clearly, $\tilde{\BO}\cup\tilde{\BO}_1\cup\tilde{\BO}_2=\BO=[n]$. Now we can replace 
$\vv$, $\dd_1,\dots,\dd_k$ with 
$\vv_{\tilde{\BO}\cup\tilde{\BO}_2}$, $\dd^{i_1},\dots,\dd^{i_p}$ in the assumptions of the lemma,  and obtain the desired $L$ for the latter.  If we have the desired $L$ for $\vv_{\tilde{\BO}\cup\tilde{\BO}_2}$, $\dd^{i_1},\dots,\dd^{i_p}$, then \eqref{inequality} will readily  follow, and we may deduce \eqref{inequality_2} upon observing that 
\begin{equation*}
\begin{aligned}
\dd^1+\eps\vv&\leq\dd^{i_1}+\vv_{\tilde{\BO}\cup\tilde{\BO}_2},\\
\dd^{i_p}+\eps(2|\tilde{\BO}|\1_{\tilde{\BO}}+(\vv_{\tilde{\BO}\cup\tilde{\BO}_2}+\1)_{[n]\setminus\tilde{\BO}})&=\dd^{i_p}+\eps(2|\tilde{\BO}|\1_{\tilde{\BO}}+\vv_{\tilde{\BO}_2}+\1_{[n]\setminus\tilde{\BO}})\\
&\leq\dd^k+\eps(2|\tilde{\BO}|\1_{\tilde{\BO}}+\1_{[n]\setminus\tilde{\BO}})\\
&\leq\dd^k+\eps 2n\1.
\end{aligned}
\end{equation*}

Hence we may assume that $|\BO|=n$ and $|W|=k>1$. We proceed by finding $s\in\N$, a sequence of non-negative integers  $n_1,\ldots,n_s \in\N$,  vectors $\vv^1,\ldots,\vv^s\in\R^n$, and sets  $\BO_1,\ldots,\BO_{s-1}\subset [n]$ and $L_1,\ldots,L_{s-1}\subset \R^{n}$  that satisfy, and are defined via, the following seven conditions:
\begin{itemize}
\item[(i)]  Set $n_1=1$ and, for $i\in[s-2]$, let $n_{i+1}=\min\{l\in[k];\ \BO(\dd^{n_i},\dd^l)=[n]\}$. We clearly have $1=n_1<n_2<\dots<n_s=k+1$. 
\item[(ii)] Define $\vv^1=\vv$ and $\vv^i=2n\1$, for $i\in\{2,\dots,s\}$. 
\item[(iii)] For $i\in[s-1]$, let $\BO_i:=\BO(\dd^{n_i},\dd^{n_{i+1}-1})$.
\item[(iv)] For every $i\in[s-2]$ the set $L_i$ is a  chain and
$\{\dd^{n_i}+\eps\vv^i\}\leq L_i\leq\{\dd^{n_{i+1}}+\eps\vv^{i+1}\}.$
\item[(v)] The set $L_{s-1}$ is a  chain and
$\{\dd^{n_{s-1}}+\eps\vv^{s-1}\}\leq L_{s-1}\leq\{\dd^{k}+\eps\vv^{s}\}.$
\item[(vi)] For every $i\in[s-2]$ we have 
$\Ha^1(A\cap L_i)\geq(1-(\alpha+2)\eps)(1-\eps)^{n}(n_{i+1}-n_i).$
\item[(vii)] $\Ha^1(A\cap L_{s-1})\geq(1-(\alpha+2)\eps)(1-\eps)^{n}(k-n_{s-1}).$
\end{itemize}
Notice that, since $k>1$ and $\BO=[n]$, we have $s\geq3$.
It remains to show how to find the sets $L_i$, for $i\in [s-1]$, such that (iv)-(vii) are satisfied.
To this end, we first construct sets $T_i, V_i\subset\R^n$ that satisfy the following:
\begin{itemize}
\item[(a)] The sets $T_i$, $V_i$ are chains, for all $i\in[s-1]$.
\item[(b)] For $i\in[s-1]$ we have 
$\{\dd^{n_i}+\eps\vv^i\}\leq T_i\leq \{\dd^{n_{i+1}-1}+\eps(2|\BO_i|\1_{\BO_i}+(\vv^i+\1)_{[n]\setminus\BO_i})\}.$
\item[(c)] For $i\in[s-2]$ we have 
$\{\dd^{n_{i+1}-1}+\eps(2|\BO_i|\1_{\BO_i}+(\vv^i+\1)_{[n]\setminus\BO_i})\}\leq V_i\leq\{\dd^{n_{i+1}}+\eps\vv^{i+1}\}.$
\item[(d)] $\Ha^1(T_i\cap A)\geq(1-(v+2)\eps)(1-\eps)^{2n+1}(n_{i+1}-n_i-1)$ for $i\in[s-1]$.
\item[(e)] $\Ha^1(V_i\cap A)\geq(1-(v+2)\eps)(1-\eps)^{2n+1}$, for $i\in[s-2]$.

\end{itemize}

We begin with  the sets $T_i$. Consider  $\dd^{n_i},\dots,\dd^{n_{i+1}-1}$ and $\vv^i$ instead of $\dd^1,\dots,\dd^k$ and $\vv$ and observe that we are in the same situation as in the first part of the proof (i.e., the case $|\BO|<n$). So, we are able to find $T_i$ satisfying (a), (b), (d).

Now, we proceed with the sets $V_i$. Put $V_{s-1}:=\emptyset$. For $i\in[s-2]$  choose some $l\in\BO(\dd^{n_{i+1}-1},\dd^{n_{i+1}})$ and use Lemma~\ref{one cube} for $2|\BO_i|\1_{\BO_i}+(\vv^i+\1)_{[n]\setminus\BO_i}$, $\eps$ and $l$. Since $\frac{1}{n+2}>\eps$, $\frac{1}{v_i+2}>\eps$ and $n_{i+1}-1\in W$ we can find $V_i$ satisfying (a), (e) and 
\begin{equation*}
\begin{aligned}
\{\dd^{n_{i+1}-1}+\eps(2|\BO_i|\1_{\BO_i}+(\vv^i+\1)_{[n]\setminus\BO_i})\}\\
\leq V_i\leq\{\dd^{n_{i+1}-1}+\e^l+\eps((2|\BO_i|+1)\1_{\BO_i}+(\vv^i+2\cdot\1)_{[n]\setminus(\BO_i\cup\{l\})})\}\\
\leq\dd^{n_{i+1}}+\eps(2|\BO_i|+1)\1_{\BO_i}\leq\{\dd^{n_{i+1}}+\eps\vv^{i+1}\}.
\end{aligned}
\end{equation*}
In the last but one inequality, we used (ii). Thus $V_i$ also satisfy (c).

Now let $L_i:=T_i\cup V_i$. Since $V_{s-1}=\emptyset$ we immediately deduce (iv) and (v) from (a), (b) and (c). We can also derive (vi) and (vii) from (d) and (e). We have thus completed the construction of the sets $L_i$. 

Finally, we put $L:=\bigcup_{i=1}^{s-1}L_i$. By (iv), (v), (vi) and (vii), it follows that $L$ is a chain satisfying \eqref{inequality} and \eqref{inequality_2}.
\end{proof}

Recall, from Section~\ref{sec:k_sperner}, the definition of cubes $Q_{\dd}$, where $\dd\in [m-1]_0^n$. 
Given a set of cubes $\mathcal{Q}=\{Q_{\dd_1},\ldots, Q_{\dd_k}\}$, we say that $\mathcal{Q}$ is a \emph{chain of $m$-cubes} if $\dd_1\le\cdots\le \dd_k$ is a chain in $[m-1]_0^n$.  

\bigskip

\begin{theorem}\label{thm_chain}

Let $n\in\N$, $\frac{1}{2n+2}>\varepsilon>0$, $A\subset\R^n$ be measurable and $\mathcal{Q}$ be a chain of $m$-cubes such that for every $Q\in\mathcal{Q}$ we have $\lambda_n(A\cap Q)>1-\varepsilon^{2n}2^{-n}$. Then there exists chain $L\subset\R^n$ such that $\mathcal{H}^1(A\cap L)\geq(1-(2n+2)\varepsilon)(1-\varepsilon)^n\frac{|\mathcal{Q}|-1}{m}$.

\end{theorem}

\begin{proof}
Since Lebesque measure is inner regular, we may assume that $A$ is compact. The result follows  immediately  from Lemma~\ref{Cubes} upon setting $\vv=\0$ and rescaling to $m$-cubes.
\end{proof}

The proof of Lemma~\ref{D_e_chain} is almost complete. 

\begin{proof}[Proof of Lemma~\ref{D_e_chain}]
Let $\dd_1\le \cdots \le \dd_t$ be a chain in $D_{\eps}$ and assume contrariwise that $t\ge m\kappa' + n +1$, where 
$\kappa' =\frac{\kappa}{(1-(2n+2)\eps) \cdot (1-\eps)^n}$.
Then $\mathcal{Q}=\{Q_{\dd_1},\ldots, Q_{\dd_t}\}$ is a chain of $m$-cubes that satisfies the hypothesis of Theorem~\ref{thm_chain} and so there exists chain a $L\subset\R^n$ such that 
\[
\Ha^1(A\cap L) > \kappa\, ,
\]
contrary to the assumption that $\Ha^1(A\cap C)\le \kappa$, for all chains $C\subset [0,1]^n$. 
The result follows. 
\end{proof}

\section{Concluding remarks}\label{conclusion}

Notice that in Theorem~\ref{cont_k_sperner} we assume that $\kappa >0$. 
However, it seems reasonable to ask what happens when $\kappa = 0$. 
In this case we are dealing with a measurable set $A\subset [0,1]^n$ which satisfies 
\begin{equation}\label{eq_concl}
\Ha^1(A\cap C) = 0, \; \text{ for all chains } \; C\subset [0,1]^n \, . 
\end{equation}
What is the maximum ``size'' of a measurable $A\subset [0,1]^n$ that satisfies~\eqref{eq_concl}? 
It is shown in~\cite{EMPR} that if $A\subset [0,1]^n$ is such that 
\begin{equation}\label{eq_concl_1}
\Ha^0(A\cap C) \le 1, \; \text{ for all chains } \; C\subset [0,1]^n \, , 
\end{equation}
then the Hausdorff dimension of $A$ is less than or equal to $n-1$, and that the bound is sharp. 
Let us remark that a set $A$ satisfying~\eqref{eq_concl_1} is referred to as an \emph{antichain}. 
An example of an antichain in the unit $n$-cube is the set 
\begin{equation}\label{antichain}
A_t = \left\{(x_1,\ldots,x_n)\in [0,1]^n : \sum_{i=1}^{n} x_i = t \right\} ,\, \text{ where } \, t\in [0,n] .
\end{equation}
Since an antichain, by definition, satisfies~\eqref{eq_concl_1}  it readily follows that it satisfies~\eqref{eq_concl}. 
This implies that there exist subsets of the unit $n$-cube that satisfy~\eqref{eq_concl} and whose Hausdorff dimension is equal to $n-1$. In fact, we can say a bit more. 
Recall that $\dim_H(\cdot)$ denotes Hausdorff dimension (see~\cite[p.~86]{Evans_Gariepy}). 

\bigskip 

\begin{theorem}\label{conj_s}
Fix $s\in [0,1]$ and $\beta\ge 0$. Then there exists $A\subset [0,1]^n$  that satisfies  
$\dim_H(A) = n-1 + s$ and 
\[
\Ha^s(A\cap C) \le \beta,  \; \text{ for all chains } \; C\subset [0,1]^n \,  .
\]
\end{theorem}
\begin{proof} 
Let $A_t, t\in [0,n]$, be as in~\eqref{antichain}. Let $K\subset [0,n]$ be a set such that $\dim_H(K)=s$ and $\Ha^s(K) = \beta$, and define the set 
\[ 
A = \bigcup_{t\in K} A_t \, . 
\] 
It follows from~\cite[Theorem~1.2]{Hera_Keleti_Mathe} that $\dim_H(A) = n-1+s$, and it remains to show the second statement. 

Let $C\subset [0,1]^n$ be a chain and consider the function 
$\theta:[0,1]^n \to [0,n]$ defined via 
\[ 
\theta(x_1,\ldots,x_n) = \sum_{i=1}^{n}x_i \, .
\]
Notice that $\theta$ restricted on $C$ is injective and therefore $\theta$ is a bijection from $C$ onto its image $\theta(C)$. Using a similar argument as in the proof of Theorem~\ref{chains}, we conclude that $\theta^{-1}: \theta(C)\to C$ is Lipschitz with constant $1$. Since $\theta(A\cap C) \subset K$ and $\theta$ is $1$-Lipschitz, we have        
\[ 
\Ha^s(A\cap C)  \le \Ha^s(\theta^{-1}(K)) \le \Ha^s(K) = \beta,
\]
as desired. 
\end{proof}

Given Theorem~\ref{conj_s},  the following problem arises naturally.  
\bigskip

\begin{problem}\label{prbl}
Fix $s\in [0,1]$ and $\beta\ge 0$.  
Let  $A\subset [0,1]^n$ be a measurable set such that $\dim_H(A) = n-1 + s$ and 
$\Ha^s(A\cap C) \le \beta$,  for all chains $C\subset [0,1]^n$. 
What is a sharp upper bound on $\Ha^{n-1+s}(A)$?
\end{problem}

In this article we considered the case $s=1, \beta \in [0,n]$ in Problem~\ref{prbl} . 
The case $s=0, \, \beta =1$ in Problem~\ref{prbl} has been considered in~\cite{EMPR} where it is shown that a set $A\subset [0,1]^n$ for which it holds 
$\Ha^0(A\cap C) \le 1$,  for all chains $C\subset [0,1]^n$, satisfies  $\Ha^{n-1}(A)\le n$, and that the bound is best possible.

\end{document}